\documentclass[12pt,reqno]{amsart}
\usepackage{amsmath,amsthm} 
\usepackage{amssymb,amsfonts,amsxtra}
\usepackage{psfrag,graphicx}
\usepackage{hyperref}

\theoremstyle{plain}
\newtheorem{theorem}{Theorem}
\newtheorem{proposition}[theorem]{Proposition}

\newtheorem{lemma}[theorem]{Lemma}

\theoremstyle{definition}
\newtheorem{definition}[theorem]{Definition}
\newtheorem{remark}[theorem]{Remark}

\newtheorem*{modelsituation}{Model Situation}

\numberwithin{theorem}{section}

\makeatletter
  \def\tagform@#1{\maketag@@@{%
   \textbf{(\ignorespaces#1\unskip\@@italiccorr)}}}%
   \renewcommand{\eqref}[1]{\textup{\maketag@@@{(\ignorespaces%
        {\ref{#1}}\unskip\@@italiccorr)}}}
\let\@itemize@\itemize
\def\itemize{\@itemize@\parskip 0pt\relax}
\def\@listi{\leftmargin28.5pt\parsep 0pt\topsep 2pt plus1pt minus1pt
 \itemsep2pt plus2pt minus1pt}
\let\@listI\@listi
\@listi
\makeatother

	{\setlength{\unitlength}{1mm}
	\begin{center}
	\begin{picture}(#1,#2)
	\scriptsize
}%
	{\end{picture}
	\end{center}
 	\noindent}

\DeclareFontFamily{OML}{rsfs}{\skewchar\font'177}
\DeclareFontShape{OML}{rsfs}{m}{n}{ <5> <6> rsfs5 <7> <8> <9> rsfs7
  <10> <10.95> <12> <14.4> <17.28> <20.74> <24.88> rsfs10 }{}
\DeclareMathAlphabet{\mathfs}{OML}{rsfs}{m}{n}

\newcommand{\co}{\colon\thinspace}
\newcommand{\Z}{\mathbb{Z}}
\newcommand{\R}{\mathbb{R}}

\newcommand{\la}{\langle}
\newcommand{\ra}{\rangle}

\DeclareMathOperator{\Fix}{Fix}
\DeclareMathOperator{\Cent}{Cent}

\begin{document}

%%%%%%%%%%%%%%%%%%%%%%%%%%%%%%%%%%%%%%%%%%%%%%%%%%%%%%%%%%%%%%%%%%%%%%%%%%%%

\title{Morse theory and conjugacy classes of finite subgroups II}

\author[N.~Brady]{Noel Brady$^1$}
\address{Dept.\ of Mathematics\\
        University of Oklahoma\\
	Norman, OK 73019}
\email{nbrady@math.ou.edu}

\author[M.~Clay]{Matt Clay}
\address{Dept.\ of Mathematics\\
        University of Oklahoma\\
	Norman, OK 73019}
\email{mclay@math.ou.edu}

\author[P.~Dani]{Pallavi Dani}
\address{Dept.\ of Mathematics\\
        Louisiana State University \\
	Baton Rouge, LA 70803}
\email{pdani@math.lsu.edu}

\date{\today}

\begin{abstract}
  We construct a hyperbolic group containing a finitely presented
  subgroup, which has infinitely many conjugacy classes of
  finite-order elements.  

  We also use a version of Morse theory with high dimensional
  horizontal cells and use handle cancellation arguments to produce
  other examples of subgroups of CAT(0) groups with infinitely many
  conjugacy classes of finite-order elements.
\end{abstract}

\maketitle
\footnotetext[1]{N.\ Brady was partially supported by NSF grant no.\
  DMS-0505707} 

%%%%%%%%%%%%%%%%%%%%%%%%%%%%%%%%%%%%%%%%%%%%%%%%%%%%%%%%%%%%%%%%%%%%%%%%%%%%

\section*{Introduction}

This paper is a continuation of our earlier work in \cite{BrCD}. In
that paper we showed how the construction of Leary--Nucinkis \cite{LN}
fits into a more general framework than that of right angled Artin
groups.  We used this more general framework to produce a CAT(0) group
containing a finitely presented subgroup with infinitely many
conjugacy classes of finite-order elements. Unlike previous examples
(which were based on right-angled Artin groups) our ambient CAT(0)
group did not contain any rank 3 free abelian subgroups.  In the
current paper (see Section~\ref{sec:hyper}), we produce a hyperbolic
group containing a finitely presented subgroup which has infinitely
many conjugacy classes of finite-order elements.

We work in the more general situation of Morse functions with high
dimensional horizontal cells in Sections~\ref{sec:fm}
and~\ref{sec:rips} of this paper. This allows us to see (in Section~\ref{sec:fm}) that the
original examples of Feighn--Mess \cite{feighn-mess} fit into the same
general framework as the examples of Leary--Nucinkis \cite{LN}. This
addresses a remark we made after Example~1.2 of~\cite{BrCD}.

In Section~\ref{sec:rips} we use Morse theory with horizontal
cells to see that a suitable modification of the Rips' construction
(suggested to the authors by Dani Wise) produces many new examples of
hyperbolic groups containing finitely generated subgroups with
infinitely many conjugacy classes of finite-order elements. The Morse
arguments used here involve a careful accounting of handle
cancellations.

%%%%%%%%%%%%%%%%%%%%%%%%%%%%%%%%%%%%%%%%%%%%%%%%%%%%%%%%%%%%%%%%%%%%%%%%%%%%

\section{Counting conjugacy classes of finite-order elements} \label{sec:conj}

The following proposition describes a general algebraic situation
which ensures that a conjugacy class in some group will intersect a
subgroup in infinitely many conjugacy classes. In all the examples in
this paper, the target group $Q$ is just $\Z$, and the result is used
to count conjugacy classes of finite-order elements.

\begin{proposition}\label{prop:conj}
  Suppose $\varphi\co G \to Q$ is an epimorphism where
  $$[\Cent_Q(\varphi(\sigma)):\varphi(\Cent_G(\sigma))] = \infty$$ for
  some element $\sigma \in G$.  Then the conjugacy class of $\sigma$
  in $G$ intersects $\varphi^{-1}(\langle \varphi(\sigma) \rangle)$ in
  infinitely many $\varphi^{-1}(\langle \varphi(\sigma)
  \rangle)$-conjugacy classes.
\end{proposition}

\begin{proof}
  Let $q_1,q_2 \in \Cent_Q(\varphi(\sigma))$ and fix $g_1,g_2 \in G$
  such that $\varphi(g_i) = q_i$.  Now $g_1\sigma g_1^{-1}$ is
  conjugate to $g_2\sigma g_2^{-1}$ in $\varphi^{-1}(\langle
  \varphi(\sigma) \rangle)$ if and only if there is an $h \in
  \varphi^{-1}(\langle \varphi(\sigma) \rangle)$ such that $hg_1\sigma
  g_1^{-1}h^{-1} = g_2\sigma g_2^{-1}$, equivalently $g_2^{-1}h g_1
  \in \Cent_G(\sigma)$.  Applying $\varphi$ we see this is equivalent
  to $q_2^{-1}\varphi(\sigma)^m q_1 \in \varphi(\Cent_G(\sigma))$ for
  some $m$.  Since $q_2 \in \Cent_Q(\varphi(\sigma))$, this is
  equivalent to $\varphi(\sigma)^m q_2^{-1}q_1 \in
  \varphi(\Cent_G(\sigma))$, in other words $ q_2^{-1}q_1 \in
  \varphi(\Cent_G(\sigma))$.

  Therefore, the conjugacy class of $\sigma$ in $G$ intersects
  $\varphi^{-1}(\langle \varphi(\sigma) \rangle)$ in at least
  $[\Cent_Q(\varphi(\sigma)):\varphi(\Cent_G(\sigma))]$ many
  $\varphi^{-1}(\langle \varphi(\sigma)\rangle)$-conjugacy classes.
  By hypothesis, this index is infinite, completing the proof.
\end{proof}

\begin{remark}\label{rk:app}
  Proposition \ref{prop:conj} generalizes Lemma 2 in \cite{Bridson00}
  where $Q$ is an infinite abelian group, and the hypothesis on the
  index of the centralizers is ensured by requiring
  $|\varphi(\Cent_G(\sigma))|< \infty$.
\end{remark}

\begin{modelsituation}
  Let $\varphi\co G \to {\mathbb Z}$ be an epimorphism where $G$ is a
  CAT(0) group.  Let $X$ be a CAT(0) metric space on which $G$ acts
  properly by isometries, and let $f\co X \to {\mathbb R}$ be a
  $\varphi$-equivariant (Morse) function,where ${\mathbb Z}$ acts on
  ${\mathbb R}$ by integer translations.  Let $\sigma \in G$ have
  finite order and the property that $f({\rm Fix}(\sigma)) \subset
  {\mathbb R}$ is compact.  This generalizes our model situation from
  \cite{BrCD} where we required that $\sigma$ had an isolated fixed
  point.  Then $g \in \Cent_G(\sigma)$ implies that $g$ acts on
  $\Fix(\sigma)$ and therefore by $\varphi$-equivariance of $f$,
  $\varphi(g)$ acts on $f(\Fix(\sigma))$.  Since $f(\Fix(\sigma))$ is
  compact and $\Z$ acts on $\R$ by translations, we see that
  $\varphi(\Cent_G(\sigma)) = 0$.  Therefore, applying Proposition
  \ref{prop:conj} to $\varphi\co G \to \Z$, the conjugacy class of
  $\sigma$ in $G$ intersects $\ker(\varphi)$ in infinitely many
  $\ker(\varphi)$-conjugacy classes, since $\varphi^{-1}(\langle
  \varphi(\sigma) \rangle) = \ker(\varphi)$.
\end{modelsituation}
  
%%%%%%%%%%%%%%%%%%%%%%%%%%%%%%%%%%%%%%%%%%%%%%%%%%%%%%%%%%%%%%%%%%%%%%%%%%%% 
  
\section{The Feighn--Mess examples}\label{sec:fm}

In this section we show how the Feighn--Mess examples
\cite{feighn-mess} fit into the more general Morse theory set-up in
Proposition \ref{prop:conj}.  This addresses the remark we made after
Example 1.2 in \cite{BrCD}.  Recall the Feighn--Mess examples are
subgroups of $(\mathbb{F}_2)^n \rtimes \langle \sigma \rangle$ where
the $\langle \sigma \rangle$ factor acts by an involution in each
$\mathbb{F}_2$ that fixes one generator and sends the other to its
inverse.

Let $Y$ be the wedge of two circles $S^1 = \R/\Z$ glued together at
the point $0$.  Label the two circles $a$ and $b$.  There is an order
two isometry $\sigma$ on $Y$ defined by the identity on $a$ and by
$\sigma(x) = 1-x$ on $b$.  The fixed point set of $\sigma$ consists of
two components: one is the circle $a$, the other is the point $p =
\frac{1}{2} \in b$.  This induces a coordinate-wise defined map
$\sigma_n\co Y^n \to Y^n$ where $Y^n$ is the direct product of $n$
copies of $Y$.  Here we see that the fixed point set of $\sigma_n$ has
$n+1$ homeomorphism types of components, a representative of each is
of the form $a_1 \times \cdots \times a_i \times p_{i+1} \times \cdots
\times p_n$, which is isometric to the $i$ dimensional torus
$(\R/\Z)^i$.
  
Define $h\co Y^n \to S^1 = \R/\Z$ by mapping each of the $n$ circles
labeled $a$ homeomorphically around $\R/\Z$, the $n$ circles labeled
$b$ to $0$ and extending linearly over the higher skeleta.  Clearly
$h$ is $\sigma_n$-equivariant.

Let $X$ be the universal cover of $Y^n$ and lift $h\co Y^n \to S^1$ to
$f\co X \to \R$.  Then $X$ is a CAT(0) cubical complex.  We can
identify $H = \pi_1(Y^n) = (\mathbb{F}_2)^n$ with a subgroup of the
group of isometries of $X$.  There are several different types of
lifts of the isometry $\sigma_n$ to an isometry
$\widetilde{\sigma_n}\co X \to X$ corresponding to the different
homeomorphism types of components in the fixed point set of
$\sigma_n$.  For each $0 \leq i \leq n$ we get a lift
$\widetilde{\sigma_{n,i}}$ whose fixed set is a lift of $a_1 \times
\cdots \times a_{i} \times p_{i+1} \times \cdots \times p_n$.  This
lift is isometric to an $i$-dimensional plane $\R^i$.  The image of
$\Fix(\widetilde{\sigma_{n,i}})$ under the map $f$ is $\R$ if $0< i
\leq n$ and a single point if $i=0$.  Let $\widetilde{\sigma_n} =
\widetilde{\sigma_{n,0}}$ be such that $f(\Fix(\widetilde{\sigma_n}))
= 0 \in \R$.  For $n=1$, we have drawn the action of
$\widetilde{\sigma_0}$ on $X$ in Figure \ref{fig:fm}, for $n >1$, the
maps $\widetilde{\sigma_n}$ are induced coordinate-wise from this map.

\begin{figure}[ht]
  \psfrag{s}{$\widetilde{\sigma_0}$}
  \psfrag{X}{$X$}
  \psfrag{R}{$\R$}
  \psfrag{f}{$f$}
  \psfrag{F}{$\Fix(\widetilde{\sigma_0})$}
  \psfrag{0}{$0$}
  \centering
  \includegraphics{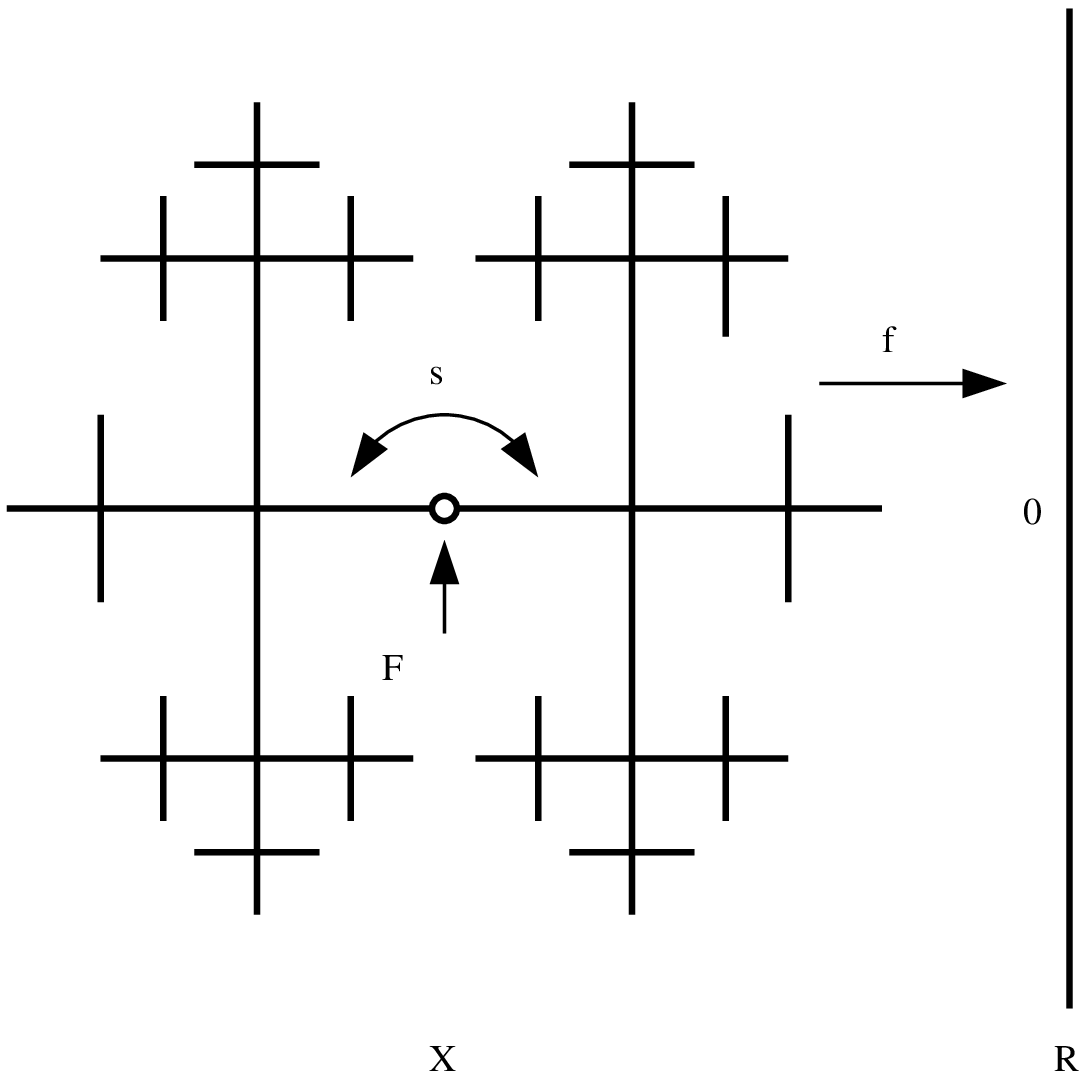}
  \caption{The isometry $\widetilde{\sigma_0}\co X \to X$.}
  \label{fig:fm}
\end{figure}

Let $G = H \rtimes \langle \widetilde{\sigma_0} \rangle$, this is the
group considered by Feighn and Mess.  Then there is a homomorphism
$\varphi \co G \to \Z$ induced from $h_*\co H \to \Z$ and by sending
$\widetilde{\sigma_0} \mapsto 0$.  The map $f\co X \to \R$ is
$\varphi$-equivariant.  This is our model situation where Proposition
\ref{prop:conj} applies, hence the conjugacy class of
$\widetilde{\sigma_0}$ in $G$ intersects $K=\ker(\varphi)$ in
infinitely many $K$-conjugacy classes.

To compute the finiteness properties of $K$ we use Morse theory with
\emph{horizontal} cells.  The map $f\co X \to \R$ is a
$h_*$-equivariant Morse function in the sense of \cite{BrMi}.  Since
finiteness properties are virtual notions, it suffices to compute the
finiteness properties of $\ker(h_*) = H \cap K$.  Any horizontal cell
for $\widetilde{\sigma_0}$ is a face of a cube of the form $b_1 \times
\cdots \times b_n$.  The ascending or descending link of an
$i$-dimensional face is $(n-i-1)$-connected as it is an
$(n-i)$-simplex, where we make the convention that $(-1)$-connected
means ``not empty'' and a $(-1)$-simplex is the empty set.
Therefore, $\ker(h_*)$ is of type F$_{n-1}$.  Furthermore, since the
ascending link of a horizontal cell of the form $b_1 \times \cdots
\times b_n$ is empty we see that $\ker(h_*)$ is not of type F$_n$.
Therefore $K$ is of type F$_{n-1}$ but not of type F$_n$.
  
%%%%%%%%%%%%%%%%%%%%%%%%%%%%%%%%%%%%%%%%%%%%%%%%%%%%%%%%%%%%%%%%%%%%%%%%%%%%

\section{Examples arising from Rips' construction}\label{sec:rips}

The idea behind following examples was suggested by Dani
Wise. Consider the Rips' construction of a non-elementary hyperbolic
group $G_0$ with $\Z$ quotient and finitely generated kernel. Wise's
suggestion is to take a quotient of $G_0$ by a power of some element
$a_1$ of this kernel. One expects to get a new short exact sequence
$$
1 \; \to \; K \; \to \; G \; \to \; \Z \; \to \; 1
$$
where $G = G_0/\langle\!\langle a_1^k \rangle\!\rangle$ is hyperbolic,
$K$ is finitely generated but not finitely presentable, and $K$ has
infinitely many conjugacy classes of elements of order $k$.  We show
that this is the case in Theorem~\ref{thm:wise} below.

We work with Wise's CAT($-1$) version of the Rips' construction, and
use handle cancellation techniques to see that the finiteness
properties of the kernel subgroups follow from Morse theory.
% We do not need to appeal to Bieri's result on infinite index normal
% subgroups of groups of cohomological dimension $2$.

A key component of Wise's version of Rips' construction is the long
word with no two-letter repetitions. We work with a slight variant of
Wise's construction in our definition of the non-elementary hyperbolic
group $G_0$.
 
\begin{definition}[Wise's long word  and the group $G_0$] 
  Given the set of letters $\{a_2, \ldots, a_m\}$, define
  $\Sigma(a_2,\ldots, a_m)$ to be the following word
$$
\; (a_2a_2 a_3 a_2 a_4 \cdots a_2 a_m)(a_3 a_3 a_4 a_3 a_5 \cdots a_3
a_m) \cdots (a_{m-1}a_{m-1} a_m)a_m.
$$ 

It is easy to see that $\Sigma(a_2,\ldots, a_m)$ is a positive word of
length $(m-1)^2$, with no two letter subword repetitions.  For $m
\geq 30$ we can, starting at the left hand side, partition
$\Sigma(a_2,\ldots, a_m)$ into at least $2m$ disjoint subwords;
$2(m-1)$ of length $14$, and two of length $13$. Construct the
positive word $W_1$ (respectively $V_1$) by adding $a_1$ as a prefix
to one of the length $13$ subwords (respectively as a suffix to the
other length $13$ subword). Call the remaining $2(m-1)$ length $14$
subwords $W_j$ and $V_j$ for $j \geq 2$.

Now define $G_0$ by  the presentation
\begin{equation}\label{eq:G0}
  G_0 \; =  \; \langle \, a_1, \ldots, a_m, t \; \, | \;  ta_jt^{-1} = W_j\, ,  \; t^{-1}a_jt = V_j \; \, ; \; 1 \leq j \leq m \, \rangle
\end{equation}
where  $\{W_j, V_j \}$ are given above.
\end{definition}

Now define the group 
\begin{equation}\label{eq:G1}
G \; = \;  G_0 /\langle \! \langle a_1^k\rangle \! \rangle
\end{equation}
where $G_0$ is defined in equation~(\ref{eq:G0}), and $k \geq 5$. 

\begin{remark}\label{rem:a1}
  Note that the group $G_0$ surjects to $\Z$, taking $t$ and $a_j$ ($j
  \geq 2$) to the identity, and $a_1$ to a generator of $\Z$.  Thus,
  the group $G$ maps onto $\Z_k$ taking $a_1$ to a generator, and all
  the other generators to the identity. This implies that the element
  $a_1$ has order $k$ in $G$.
\end{remark}

\begin{theorem}\label{thm:wise}
  The group $G$ defined by~(\ref{eq:G1}) is CAT(-1).  Let $K$ denote
  the kernel of the map $G \to \Z$ which takes each $a_j$ to the
  identity, and $t$ to a generator of $\Z$.  Then $K$ is finitely
  generated but not finitely presentable. Furthermore, the conjugacy
  class of the element $a_1$ in $G$ intersects $K$ in infinitely many
  conjugacy classes.
\end{theorem}

\begin{proof} The proof is presented in several steps. First, we
  establish that the groups $G_0$ and $G$ are CAT(-1). Then we define
  the $G$-equivariant Morse function, and show how the conjugacy
  result follows from Section~\ref{sec:conj}. Finally, we use the
  Morse theory and an analysis of handle cancellations to establish
  the finiteness properties of $K$.

\medskip
\noindent
\textit{Step 1. The ${\rm CAT}(-1)$ structure for $G_0$.} 
First one sees that the group $G_0$ is CAT($-1$), by subdividing each
relator $2$-cell in its presentation $2$-complex into $5$
right-angled hyperbolic pentagons as shown in Figure~\ref{fig:W}.  The
presentation $2$-complex of $G_0$ satisfies the large link condition.
This is a consequence of the fact that the $V_j$'s and the $W_j$'s are
positive words with no two-letter repetitions. The argument is
identical to that of \cite{Wise}. We highlight the key points for the
reader's convenience.

\begin{figure}[ht]
  \centering
  \psfrag{t}{$t$}
  \psfrag{a}{$a_j$}
  \psfrag{W}{$W_j$}
  \psfrag{V}{$V_j$}
  \includegraphics{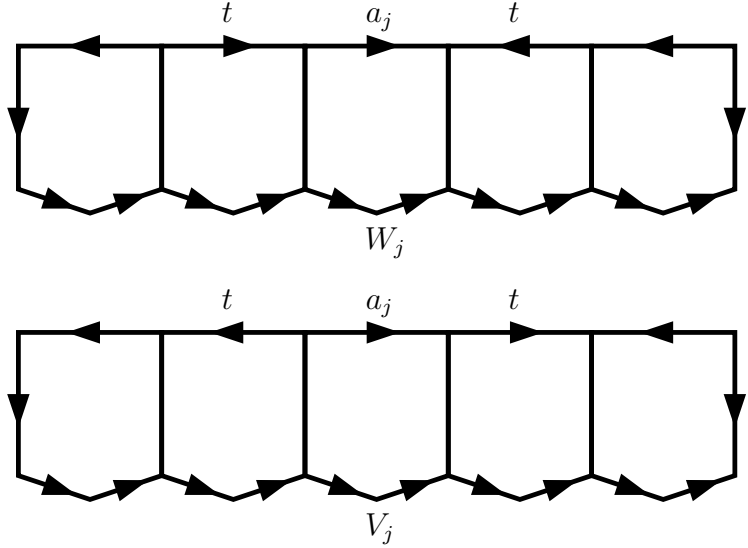}
  \caption{The subdivision of the 2-cells into right-angled pentagons.}
  \label{fig:W}
\end{figure}

The link is obtained from a subgraph on the $\{a_i^{\pm}\}$ by adding
the vertices $t^{\pm}$ and new edges of length $\pi$ connecting the
$t^{\pm}$ to the $a_i^{\pm}$. This is large if and only if the
subgraph on the $a_i^{\pm}$ is large.  The latter graph is large
because it is bipartite (since the $V_i$ and the $W_i$ are all
positive words), has no bigons (since the collection of all $V_i$ and
$W_i$ have no two letter repetitions by design) and the edge lengths
are all at least $\pi/2$.

\medskip
\noindent
\textit{Step 2. The ${\rm CAT}(-1)$ structure for $G$.}  
Next, we show that $G = G_0/ \langle \! \langle a_1^k\rangle \!
\rangle$ is a CAT($-1$) group for $k \geq 5$. First attach a $2$-cell
$\Delta$ to the loop $a_1^k$ in the locally CAT($-1$) presentation
$2$-complex for $G_0$.  The resulting complex $Y$ is a presentation
$2$-complex for $G$.

By Remark~\ref{rem:a1}, the preimage of the loop $a_1$ in the
universal cover $\widetilde Y$ of $Y$ consists of a disjoint
collection of embedded circles of length $k$ (of the form
$a_1^k$). The preimage of $\Delta$ in $\widetilde Y$ consists of
distinct families of $k$ $2$-cells, one for each component of the
preimage of $a_1$.  Each of the $k$ $2$-cells in a family is attached
to the same embedded loop $a_1^k$.

For each loop labeled by $a_1^k$ in $\widetilde Y$ collapse the $k$
attached $2$-cells in the preimage of $\Delta$ down to a single
$2$-cell. The resulting cell complex $X$ can be given a locally
CAT($-1$) structure, by giving each $2$-cell in the preimage of
$\Delta$ the geometry of a regular hyperbolic $k$-gon whose side
lengths are equal to the side length of a right-angled hyperbolic
pentagon.  The proof that $X$ is locally CAT($-1$) involves a slight
modification of the argument given above to show that $G_0$ is
CAT($-1$).  In particular, we add a single edge from $a_1^{+}$ to
$a_1^{-}$ of length at least $\pi/2$, and note that the subgraph on
the $\{ a_i^{\pm} \}$ is still bipartite with no bigons.

Since $X$ is $1$-connected and locally CAT($-1$), it is CAT($-1$),
and since $G$ acts properly discontinuously and co-compactly on $X$,
we conclude that $G$ is a CAT($-1$) group. Hence $G$ is hyperbolic.

\begin{figure}[ht]
  \psfrag{ai}{$a_i$}
  \psfrag{G}{$\Gamma_i$}
  \psfrag{t}{$t$}
  \psfrag{as1}{$a_{\sigma_i(1)}$}
  \psfrag{as2}{$a_{\sigma_i(2)}$}
  \psfrag{as14}{$a_{\sigma_i(14)}$}
  \psfrag{at1}{$a_{\tau_i(1)}$}
  \psfrag{at2}{$a_{\tau_i(2)}$}
  \psfrag{at14}{$a_{\tau_i(14)}$}
  \psfrag{r1}{$r_{i,1}$}
  \psfrag{r2}{$r_{i,2}$}
  \psfrag{r13}{$r_{i,13}$}
  \psfrag{r14}{$r_{i,14}$}
  \psfrag{s1}{$s_{i,1}$}
  \psfrag{s2}{$s_{i,2}$}
  \psfrag{s3}{$s_{i,3}$}
  \psfrag{s14}{$s_{i,14}$}
  \psfrag{d1}{$\Delta_{i,1}$}
  \psfrag{d2}{$\Delta_{i,2}$}
  \psfrag{d14}{$\Delta_{i,14}$}
  \psfrag{c}{$\cdots$}
  \centering
  \includegraphics{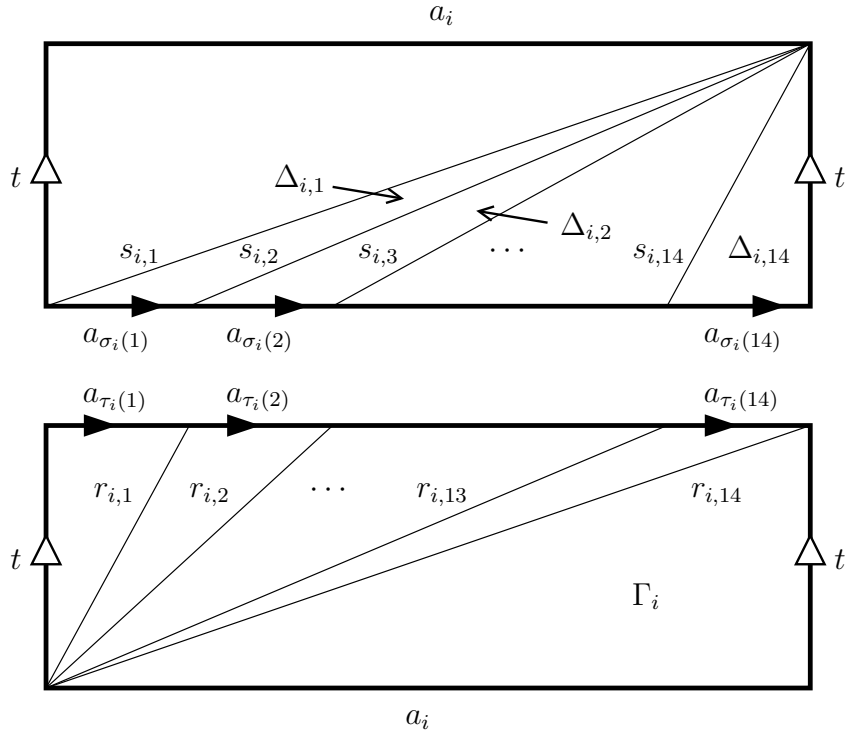}
  \caption{The subdivided 2-cells in $X$.}\label{fig:cells}
\end{figure}
 
\medskip
\noindent
{\em Step 3. The Morse function and counting conjugacy classes.}  
We define a circle valued Morse function on the original
(unsubdivided) presentation $2$-complex for $G$ as follows.  Send the
$t$ edge homeomorphically around the circle, and send the remaining
edges in the $1$-skeleton (namely the $a_j$) to the base-point of the
target circle. Extend linearly over the $2$-cells. Thus the $2$-cell
$\Delta$ is mapped to the base point of the target circle, and the
remaining $2$-cells are mapped onto the $t$ edge (horizontal
projection in Figure~\ref{fig:cells}) and then to the circle as
described above.

This lifts to a $G$-equivariant Morse function $f\co\widetilde{Y} \to
\R$, which factors through $X$. We denote the $G$-equivariant factor
map $X \to \R$ by $f$. Note that the $1$-cells labeled by $a_j$ are
all horizontal, and that the $2$-cells of $X$ in the preimage of
$\Delta$ are also horizontal.

The element $a_1$ fixes a unique point of $X$: namely, the center of
the $k$-gon with vertices $1, a_1, \ldots, a_1^{k-1}$. We are now in
the model situation of Section~\ref{sec:conj} above, and so the
conjugacy class of $a_1$ in $G$ intersects $K$ in infinitely many
$K$-conjugacy classes.

\medskip
\noindent
{\em Step 4. Finiteness properties of the kernel $K$.}  
Finally, we use Morse theory on the complex $X$ to show that $K$ is
finitely generated but not finitely presentable.

To this end, we first subdivide the $2$-cells $ta_it^{-1}W_i^{-1}$
and $t^{-1}a_itV_i^{-1}$ into triangular cells. This subdivision is
indicated in Figure~\ref{fig:cells}.  Writing $W_i =
a_{\sigma_i(1)}\ldots a_{\sigma_i(14)}$, define new edges $s_{i,j}$
for $1 \leq j \leq 14$ inductively by
$$
ta_i = s_{i,1} \quad \quad {\rm{and}} \quad \quad  s_{i,j} = a_{\sigma_i(j)}s_{i,j+1} \quad \quad 
1\leq j \leq 13. 
$$ 
For each $1 \leq i \leq m$ and each $1 \leq j \leq 14$, let
$\Delta_{i,j}$ denote the triangular $2$-cell in the subdivision of
$ta_it^{-1}W_i^{-1}$ which has boundary
$a_{\sigma_i(j)}s_{i,j+1}s_{i,j}^{-1}$.

Similarly, writing $V_i = a_{\tau_i(1)}\ldots a_{\tau_i(14)}$, define
new edges $r_{i,j}$ for $1 \leq j \leq 14$ inductively by
$$
r_{i,1} = ta_{\tau_i(1)} \quad \quad {\rm{and}} \quad \quad r_{i,j+1} = r_{i,j}a_{\tau_i(j+1)} 
\quad \quad 1\leq j \leq 13. 
$$
For each $1 \leq i \leq m$, let $\Gamma_i$ denote the triangular
$2$-cell in the subdivision of $t^{-1}a_itV_i^{-1}$ with boundary
$a_i t r_{i,14}^{-1}$.

We are in the setting of \cite{BrMi}, where $f$ has horizontal
$0$-cells, $1$-cells and $2$-cells.  The Morse theory argument will
be essentially that of \cite{BrMi}, but we will have to take care of
handle cancellations. Note that for any cell of $X$, its image under
$f$ will either be an integer or an interval of length one (between
successive integers) in $\R$.

\medskip
\noindent {\em Ascending links.} A schematic of the ascending link of
a vertex $v \in X$ is given in Figure~\ref{fig:ascendinglink}. It
consists of a graph with $(13m + 1)$ components.  One component is a
graph which is a subdivision of the cone on the $2m$ points $a_i^{-}$
and $s_{i,1}$ for $1 \leq i \leq m$. The cone vertex is $t^{-}$, and
each of the $m$ edges from $t^{-}$ to $a_i$ ($1 \leq i \leq m$) is
subdivided into $14$ segments by $r_{i,j}^{-}$ for $1 \leq j \leq 14$.
The remaining $13m$ points are labeled by $s_{i,j}^{-}$ ($1\leq i \leq
m$ and $2 \leq j \leq 14$).

\begin{figure}[ht]
  \centering
  \includegraphics{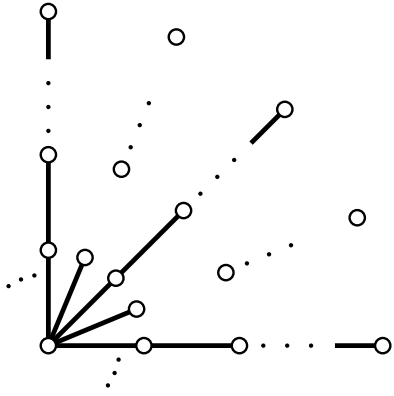}
  \caption{The ascending link of $v$.}
  \label{fig:ascendinglink}
\end{figure}

\medskip
\noindent {\em The handle additions.}  For integers $a < b$ we obtain
$f^{-1}([a,b])$ from $f^{-1}([a+1, b])$ by a succession of three types
of coning operations (equivalently, handle addition operations).
 
 \begin{enumerate}
 \item For each $0$-cell $v \in f^{-1}(a)$, let $U_v$ denote the
   union of all the cells of $X$, each of which contains $v$ and which
   maps to the interval $[a,a+1]$ under $f$. Then $U_v \cap
   f^{-1}(a+1)$ is a geometric realization of $Lk_\uparrow(v, X)$.
 
   Attaching the $U_v$ to $f^{-1}([a+1,b])$ for each $0$-cell $v \in
   f^{-1}(a)$ is equivalent to coning off each $U_v \cap f^{-1}(a+1)$
   copy of $Lk_\uparrow(v, X)$ to the corresponding vertex $v$. Denote
   the resulting complex by $X_1$.

   Now $Lk_\uparrow (v, X)$ has $(13m + 1)$-components, each of which
   are contractible. Thus coning the copy of $Lk_\uparrow (v, X)$ in
   $f^{-1}(a+1)$ off to $v$ is equivalent to attaching a wedge of
   $13m$ one-handles to $f^{-1}([a+1,b])$.  Thus $X_1$ is obtained
   from $f^{-1}([a+1,b])$ by attaching an infinite family of distinct
   wedges of $13m$ one-handles, indexed by the $0$-cells of
   $f^{-1}(a)$.
  
 \item For each $1$-cell $e$ in $f^{-1}(a)$, let $U_e$ denote the
   union of cells, each of which contains $e$ and which maps to
   $[a,a+1]$ under $f$.  Then $U_e \cap f^{-1}(a+1)$ is a geometric
   realization of $Lk_\uparrow(e, X)$. This is a discrete set of
   points.  We see that its cardinality is at least one, by writing $e
   = a_i$ for some $1 \leq i \leq m$, and noting that $Lk(a_i,
   \Gamma_i)$ is a subset of $Lk_\uparrow(e, X)$.  Recall that
   $\Gamma_i$ is the simplex containing the bottom edge $a_i$ in the
   subdivision of $t^{-1}a_itV_i^{-1}$ in Figure~\ref{fig:cells}.
 
   Note also that $U_e \cap X_1$ is isomorphic to $Lk_\uparrow(e, X)
   \ast \partial e$.  Thus, attaching $U_e$ to $X_1$ is equivalent to
   coning $Lk_\uparrow(e, X) \ast \partial e$ off to the barycenter
   $\widehat{e}$ of $e$. Since $\partial e$ has two points, this is
   equivalent to attaching a wedge of $(|Lk_\uparrow (e, X)| - 1)$
   two-handles to $X_1$. This even makes sense in the case that
   $Lk_\uparrow(e, X)$ has only one point. Then the coning operation
   is a homotopy equivalence, which is equivalent to attaching 0
   two-handles to $X_1$.
 
   Let $X_2$ denote the result of attaching $U_e$ to $X_1$ for each
   $2$-cell $e$ in $f^{-1}(a)$. As in the previous section, $X_2$ is
   obtained from $X_1$ by attaching an infinite family of wedges of
   two-handles to $X_1$, indexed by $1$-cells $e$ in $f^{-1}(a)$.
 
 \item Finally, attach the $2$-cells $d \subset f^{-1}(a)$ to $X_2$
   to obtain $f^{-1}([a,b])$.  Note that $d \cap X_2 = \partial d$ for
   each such $2$-cell, and so attaching $d$ is equivalent to coning
   off $\partial d$ to the barycenter $\widehat{d}$ of $d$.
  
 \end{enumerate}

\medskip
\noindent {\em The handle cancellations.} We show that a subset
(described in Definition~\ref{def:canc} below) of the two-handles
from the coning operation~(2) above {\em cancel} with the collection
of all one-handles from the coning operation~(1).  Note that $X_1$ is
the result of attaching all the one-handles to $f^{-1}([a+1,b])$.  By
{\em cancel} we simply mean that the space obtained from $X_1$ by
attaching this subset of two-handles is homotopy equivalent to
$f^{-1}([a+1,b])$.

\begin{definition}[Canceling set] \label{def:canc} For each $1 \leq
  \ell \leq m$ and for each horizontal edge $a_\ell \subset
  f^{-1}(a)$, consider the union $U_{a_\ell}$ of $2$-cells of $X$,
  each of which contains $a_\ell$ and which maps to $[a,a+1]$ under
  $f$. Let $U'_{a_\ell}$ denote the subcomplex of $U_{a_\ell}$ such
  that $Lk(a_\ell, U'_{a_\ell})$ is isomorphic to
$$
   Lk (a_\ell, \Gamma_\ell) \; \cup \; \bigcup_{\sigma_i(j)=\ell,\,  j\not=1} 
Lk (a_{\sigma_i(j)}, \Delta_{i,j}) . 
$$
That is, we are not considering contributions to $Lk_\uparrow(a_\ell,
X)$ which correspond to occurrences of $a_\ell$ as the first letter in
any $W_i$.  The canceling set is defined as the following subcomplex
of $f^{-1}([a,a+1])$
$$
U \; = \; \bigcup_{1\leq \ell \leq m, \, a_\ell \subset f^{-1}(a)} U'_{a_\ell} . 
$$
The shaded portion of Figure~\ref{fig:retraction} shows the
intersection of a $2$-cell $ta_it^{-1}W_i^{-1}$ in $f^{-1}([a,a+1])$
with $X_1 \cup U$.
\end{definition}

\begin{lemma}[Handle cancelation for ascending links] \label{lem:canc}
  Let $f: X \to \R$ be as defined above.  Given integers $a < b$, let
  $X_1$ be the space obtained from $f^{-1}([a+1,b])$ by attaching
  $1$-handles as described in coning operation (1) above, and let $U$
  be the canceling set of Definition~\ref{def:canc}.
 
 Then $X_1 \cup U$ is homotopy equivalent to $f^{-1}([a+1,b])$. 
 \end{lemma}

 \begin{proof} The homotopy equivalence is fairly easy to see. First,
   consider relator cells of the form $ta_it^{-1}W_i^{-1} \subset
   f^{-1}([a,a+1])$.  Figure~\ref{fig:retraction} shows the
   intersection of $X_1 \cup U$ with one of these relators.  The only
   cell of this relator which does not belong to $X_1 \cup U$ is the
   unshaded (open) $2$-cell labeled $\Delta_{i,1}$. There is an
   obvious deformation retraction of this intersection onto the
   boundary $1$-skeleton $ta_it^{-1}W_i^{-1}$. Perform all these
   deformation retractions for each $a_i\subset f^{-1}(a+1)$ and each
   $1 \leq i \leq m$, to get the space $Z$.

\begin{figure}[ht]
  \psfrag{ai}{$a_i$}
  \psfrag{as1}{$a_{\sigma_i(1)}$}
  \psfrag{as2}{$a_{\sigma_i(2)}$}
  \psfrag{as3}{$a_{\sigma_i(3)}$}
  \psfrag{as14}{$a_{\sigma_i(14)}$}
  \psfrag{t}{$t$}
  \psfrag{c}{$\cdots$}
  \centering
  \includegraphics{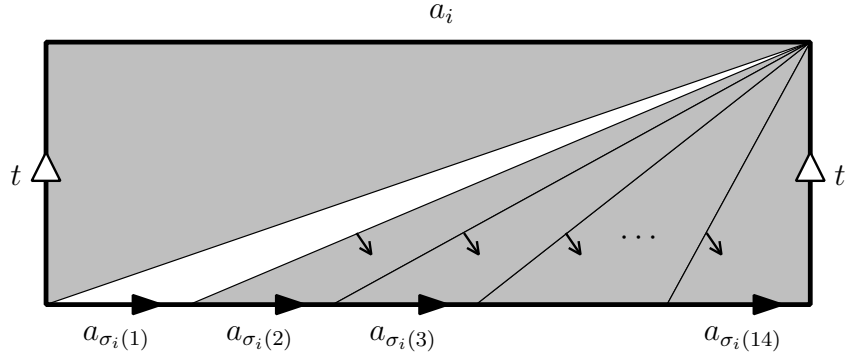}
  \caption{Step one of the deformation retraction in Lemma~\ref{lem:canc}.}
  \label{fig:retraction}
\end{figure}

Now, we turn our focus to horizontal $1$-cells at height $a$ in the
set $Z$.  Each $a_j \subset f^{-1}(a)$ is contained in just one
$2$-cell of $Z$. This $2$-cell is labeled $\Gamma_j$. Push across
the free edge $a_j$ to deformation retract the relator $t^{-1}V_j
ta_j^{-1}$ onto the subword $t^{-1}V_j t$ of its boundary word. Do
this equivariantly for all $1 \leq j \leq m$ and all $a_j \subset
f^{-1}(a)$.  The resulting space can now be deformed to
$f^{-1}([a+1,b])$ by collapsing the edges labeled $t$ in
$f^{-1}([a,a+1])$ onto their vertices at level $a+1$.
\end{proof}
 
 \medskip
 
 \noindent
 {\em Continue with the usual Morse argument.}  We obtain
 $f^{-1}([a,b])$ from $f^{-1}([a+1,b])$ by first attaching all the new
 cells of $X_1 \cup U$. Lemma~\ref{lem:canc} ensures that $X_1 \cup U$
 is homotopy equivalent to $f^{-1}([a+1,b])$. Now we obtain
 $f^{-1}([a,b])$ from $X_1 \cup U$ by attaching all $2$-cells of
 $f^{-1}([a,a+1])$ which are labeled by $\Delta_{i,1}$ for $1 \leq i
 \leq m$, and by attaching all $2$-cells of $f^{-1}(a)$. The boundary
 of each of these $2$-cells is contained in $X_1 \cup U$, and so each
 such attachment can be viewed as a $2$-handle attachment (coning
 boundary off to barycenter).

 In a similar fashion (working with descending links) one can argue
 that $f^{-1}([a,b+1])$ is obtained from $f^{-1}([a,b])$ up to
 homotopy by only attaching $2$-handles.

 Now, the usual Morse theory arguments of \cite{BeBr} apply to
 conclude that the level set $f^{-1}(0)$ is connected, and hence that
 $K$ is finitely generated.  Since we are attaching two-handles, the
 inclusion-induced map $H_1(f^{-1}([-n,n])) \to H_1(f^{-1}([-m,m]))$
 for any integers $0 < n < m$ is always a surjection. If this
 inclusion-induced homomorphism were ever the zero homomorphism, then
 we would have $H_1(f^{-1}([-m,m])) = 0$.  The two-handles attached
 to obtain $f^{-1}([-m-1, m+1])$ would produce nontrivial $2$-cycles
 in $f^{-1}([-m-1, m+1]) \subset X$ which contradicts the fact that
 $X$ is a contractible $2$-complex. Now Theorem~2.2 of \cite{brown}
 (taking $\alpha \in {\mathbb N}$ and taking $X_\alpha$ to be
 $f^{-1}([-\alpha, \alpha])$) implies that $K$ is not FP$_2$. In
 particular, $K$ is not finitely presented.
\end{proof}

%%%%%%%%%%%%%%%%%%%%%%%%%%%%%%%%%%%%%%%%%%%%%%%%%%%%%%%%%%%%%%%%%%%%%%%%%%%%

\section{Conjugacy classes in finitely presented subgroups of
  hyperbolic groups}
\label{sec:hyper}

In this section we construct a hyperbolic group with a finitely
presented subgroup which has infinitely many conjugacy classes of
finite-order elements.  Our construction is a modification of the
construction in \cite{Br} of a hyperbolic group with a finitely
presented subgroup which is not hyperbolic.  The hyperbolic group in
\cite{Br} is torsion-free, and does not admit any obvious finite-order
automorphisms.

\begin{theorem}\label{th:fp-hyp}
  There exist hyperbolic groups containing finitely presented
  subgroups which have infinitely many conjugacy classes of finite
  order elements.
\end{theorem}

As indicated above, the proof consists in constructing a variation of
the branched cover complex in \cite{Br}. The variation will have an
added symmetry that the construction in \cite{Br} lacked. This
symmetry will enable one to extend the groups under consideration by a
finite-order automorphism, and to apply
Proposition~\ref{prop:conj}. An overview of the branched cover which
parallels the construction in \cite{Br} is provided in
subsection~\ref{subsec:cover} below, and the proof that this branched
cover does indeed admit a symmetry is given in
subsection~\ref{subsec:isometry}.

%%%%%%%%%%%%%%%%%%%%%%%%%%%%%%%%%%%%

\subsection{\texorpdfstring{The branched cover $Y$}{The branched cover
    Y}} \label{subsec:cover}

Let $\Theta$ be the graph in Figure~\ref{fig:theta}, where each edge
is isometric to the unit interval.  Then $\Theta^3$, with the product
metric, is a piecewise Euclidean cubical complex of non-positive
curvature. The hyperbolic group in \cite{Br} is defined to be the
fundamental group of a particular branched cover $Y$ of $\Theta^3$
with branching locus
$$
L \; = \; (\Theta \times \{0\} \times \{1\}) \; \cup \; (\{1\} \times \Theta
\times \{0\}) \; \cup \; (\{0\} \times \{1\} \times \Theta)\, .
$$

\begin{figure}[ht]
  \psfrag{0}{$0$}
  \psfrag{1}{$1$}
  \psfrag{a}{$a$}
  \psfrag{b}{$b$}
  \psfrag{c}{$c$}
  \psfrag{d}{$d$}
  \centering
  \includegraphics[width=6cm]{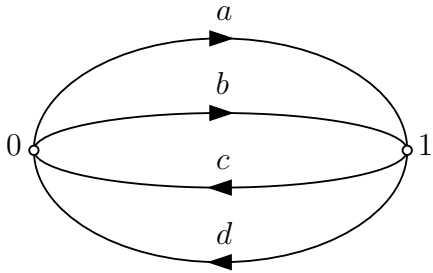}
  \caption{The graph $\Theta$.}\label{fig:theta}
\end{figure}

We refer to Section~5 of \cite{Br} for background on branched covers.
Recall from \cite{Br} that the branched cover $Y$ is obtained by the
following procedure.
\begin{itemize}
\item Remove the branching locus $L$ from the non-positively curved
  cubed complex $\Theta^3$. The resulting space $\Theta^3 - L$ has an
  incomplete piecewise Euclidean metric.
\item Take a finite cover of $\Theta^3 -L$.  The piecewise Euclidean
  metric on $\Theta^3 -L$ lifts to an incomplete piecewise Euclidean
  metric on this cover.
\item Complete this metric to form the branched cover $Y \to \Theta^3$. 
\end{itemize}
In Lemma 5.5 of \cite{Br} it is shown that if $X$ is a non-positively
curved piecewise Euclidean cubical complex and if $L \subset X$ is any
reasonable branching locus, then every finite branched cover of $X$
over $L$ is itself a non-positively curved piecewise Euclidean cubical
complex. In particular, every finite branched cover $Y$ of $\Theta^3$
over $L$ above is a non-positively curved piecewise Euclidean cubical
complex.

\smallskip

Now we define a Morse function on $\Theta^3$, and hence on finite
branched covers $Y$ of $\Theta^3$.  There is a map $\Theta \to S^1$
which maps the vertices $0$ and $1$ to a base vertex of $S^1$, and
maps the edges isometrically around the circle, with orientations
specified by the arrows in Figure~\ref{fig:theta}.  This defines a map
$\Theta^3 \to S^1\times S^1 \times S^1$. Composition with the standard
linear map $S^1\times S^1 \times S^1 \to S^1$ (the one covered by
$\R^3 \to \R : (x,y,z) \mapsto x+y+z$) gives a map $\Theta^3 \to
S^1$. Finally, the composition $Y \to \Theta^3 \to S^1$ is a
circle-valued Morse function on the branched cover $Y$.

The majority of the work in \cite{Br}
comes from constructing a specific finite branched cover $Y$ having
the following properties.
\begin{description}
\item[{\bf Hyperbolic}] $\pi_1(Y)$ is a hyperbolic group, and
\item[{\bf F$_2$--not--F$_3$}] the kernel of the map $\pi_1(Y) \to
  \pi_1(S^1)$ is finitely presented, but not of type F$_3$.
\end{description}
These properties are guaranteed by items (3) and (4) of Theorem~6.1 of
\cite{Br}.  The construction of the branched cover $Y$ is described in
detail in the proof of Theorem~6.1. We sketch the main points below,
and indicate our variation on that construction.  The key point is
that our variation is equivariant with respect to a particular
isometry $\sigma \co \Theta \to \Theta$ (see Section~\ref{subsec:isometry}) and so the branched cover $Y$ admits an
isometry $\eta$ induced by $\sigma$.

\begin{enumerate}

\item[(i)] For each $1 \leq i \leq 3$ there are projection maps
$$
{\rm pr}_i\co \Theta^3 \; \to \; \Theta^2 : (x_1, x_2, x_3) \;
\mapsto \; (x_{i+1}, x_{i+2})
$$ 
where $i$ is taken modulo $3$, and we use $3$ in place of $0$.

These restrict to projection maps
$$
{\rm pr}_i\co \Theta^3 - L \; \to \; \Theta^2 - \{(0,1)\}
$$
for $1 \leq i \leq 3$.

\item[(ii)] Let $\Delta$ be the graph (depicted in Figure~\ref{fig:delta})
$$\Theta \times \{0\} \; \cup \; \{1\} \times \Theta \subset 
\Theta^2 - \{(0,1)\}.$$

\begin{figure}[ht]
  \psfrag{(0,0)}{{\scriptsize $(0,0)$}} 
  \psfrag{(1,0)}{{\scriptsize $(1,0)$}}
  \psfrag{(1,1)}{{\scriptsize $(1,1)$}} 
  \psfrag{a0}{$a_0$} 
  \psfrag{b0}{$b_0$}
  \psfrag{c0}{$c_0$} 
  \psfrag{d0}{$d_0$} 
  \psfrag{a1}{$a_1$}
  \psfrag{b1}{$b_1$} 
  \psfrag{c1}{$c_1$} 
  \psfrag{d1}{$d_1$} 
  \centering
  \includegraphics[width=11.5cm]{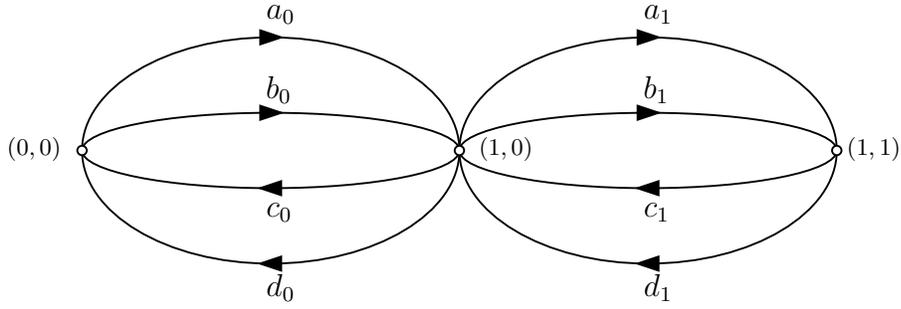}
  \caption{The graph $\Delta$.}\label{fig:delta}
\end{figure}

\noindent By Lemma~6.2 of \cite{Br}, the inclusion of $\Delta$ in
$\Theta^2 - \{(0,1)\}$ is a homotopy equivalence.  We choose the
following basis for $\pi_1(\Delta, (1,0))$, which is a free group of
rank $6$:
$$\bar{a}_0 b_0,\; \bar{a}_0\bar{c}_0,\; 
\bar{a}_0\bar{d}_0, \; a_1\bar{b}_1, \;a_1c_1,\; a_1d_1$$
 
\item[(iii)] Define a homomorphism $\rho\co \pi_1(\Delta, (0,1)) \to S_5$ (when
  convenient, we shall think of $\rho$ as a homomorphism $\rho\co
  \pi_1(\Theta^2 - \{(0,0)\},(0,1)) \to S_5$) by
  \begin{align*}
    \bar{a}_0b_0 &\mapsto \alpha^2 & a_1\bar{b}_1 & \mapsto
    \beta^2 \\
    \bar{a}_0\bar{c}_0 &\mapsto \alpha^3 & a_1c_1 & \mapsto
    \beta^3 \\
    \bar{a}_0\bar{d}_0 &\mapsto \alpha & a_1d_1 & \mapsto \beta
  \end{align*}
  where $\alpha$ and $\beta$ are the permutations
  \begin{equation*}
    \alpha = (1)(2\;5\;3\;4)\quad \text{and} \quad \beta = 
    (1\;2\;3\;4)(5) .
  \end{equation*}
  The reader can verify that $[\alpha^i,\beta^j]$ is a $5$-cycle for
  $1 \leq i,j  \leq 3$. This implies that the images under $\rho$ of the
  36 commutators obtained by taking a pair of loops in $\Delta$ (one
  composed of two $1$-cells with endpoints $(0,0)$ and $(1,0)$, and
  the second composed of two $1$-cells with endpoints $(1,0)$ and
  $(1,1)$) are all $5$-cycles.

  Thus the homomorphism $\rho$ satisfies Property~1 in the proof of
  Theorem~6.1 of \cite{Br}. See Remark~\ref{rk:property}.

\item[(iv)] For $1 \leq i \leq 3$ define homomorphisms $\rho_i\co
  \pi_1(\Theta^3 - L, (0,0,0)) \to S_5$ by $$\rho_i = \rho \circ \iota
  \circ {\rm{pr}}_i$$ where $\iota \co \pi_1(\Theta^2 -
  \{(0,0)\},(0,0)) \to \pi_1(\Theta^2 - \{(0,0) \},(1,0))$ is a change
  of basepoint isomorphism.  Specifically, $\iota$ is the isomorphism
  induced by conjugating an edge path in $\Delta$ based at $(0,0)$ by
  the edge $\bar{a}_0$.  For each $i$, this gives an action of
  $\pi_1(\Theta^3 - L, (0,0,0))$ on the set $S=\{1, 2,3,4,5\}$. Define
  an action of $\pi_1(\Theta^3 - L, (0,0,0))$ on the set $S \times S
  \times S$ via
$$
\tau = (\rho_1, \rho_2 , \rho_3)\co \pi_1(\Theta^3 - L,
(0,0,0) ) \to S_{125}.
$$  
The $125$-fold cover of $\Theta^3 - L$ associated to $\tau$ is
connected.  Lift the local metric from $\Theta^3 - L$ to $Y$, and
complete it to obtain the branched covering $b\co Y \to \Theta^3$.
\end{enumerate}

\begin{remark}\label{rk:property}
  Our version  of $\rho$ satisfies the key
  Property~1 of Theorem~6.1 of \cite{Br}. Thus items (3) and (4) of Theorem~6.1 hold, 
  and so the branched covering
$Y$ will satisfy properties {\bf Hyperbolic} and {\bf
  F$_2$--not--F$_3$} above.

However, our choice of $\rho$ has an extra symmetry built in that
  the one in \cite{Br} lacked. This symmetry will be crucial  in
  Section~\ref{subsec:isometry} below.
\end{remark}

%%%%%%%%%%%%%%%%%%%%%%%%%%%%%%%%%%%%

\subsection{\texorpdfstring{The isometry $\eta\co Y \to Y$}{The
    isometry Y -> Y}} \label{subsec:isometry}

Let $\sigma\co \Theta \to \Theta$ denote the isometry which fixes
vertices $0$ and $1$ and transposes the edges $a$ and $b$, and the
edges $c$ and $d$. This induces coordinate-wise defined isometries
$$
\sigma_2\co \Theta^2 \; \to\; \Theta^2\; : (x,y) \; \mapsto \;
(\sigma(x), \sigma(y))
$$ 
and 
$$
\sigma_3\co \Theta^3 \; \to\; \Theta^3\; : (x,y,z) \; \mapsto \;
(\sigma(x), \sigma(y), \sigma(z))\, .
$$
These restrict to isometries of the incomplete spaces
$\sigma_2\co \Theta^2 - \{(0,1)\} \; \to \; \Theta^2 - \{(0,1)\}$ 
and 
$\sigma_3\co \Theta^3 - L \; \to\; \Theta^3 - L$.

Then $\sigma_2 \co \Theta^2 - \{(0,1)\} \; \to \; \Theta^2 -
\{(0,1)\}$ induces a map $\sigma_2 \co \Delta \to \Delta$.  The action
of $\sigma_2$ on $\Delta$ is given by $a_i \leftrightarrow b_i$, $c_i
\leftrightarrow d_i$, $i=0,1$.  Moreover, for $1\leq i \leq 3$, we
have
\begin{equation}\label{eq:projection} 
  {\rm pr}_{i\ast} \circ \sigma_{3\ast} = \sigma_{2\ast} \circ {\rm pr}_{i\ast}.
\end{equation}

Let $C_\alpha \co S_5 \to S_5$ be the inner automorphism  $x
\mapsto \alpha x \alpha^{-1}$. \\
{\bf Claim.} The following two  maps $\pi_1(\Theta^2 - \{(1,0)\},(0,0)) \to S_5$ are 
equal:
\begin{equation}\label{eq:sigma2}
  \rho \circ \iota \circ \sigma_{2\ast} = C_\alpha^2 \circ \rho \circ \iota
\end{equation}

Indeed, using the fact  that $\alpha,\beta$ are
$4$-cycles, we can check equality on a free basis as follows: 
\begin{align*}
  \rho\iota\sigma_{2\ast}(b_0\bar{a}_0) & = \rho\iota(a_0\bar{b}_0) =
  \rho(\bar{b}_0a_0) = \alpha^{-2} \\
  & = \alpha^2 = C_\alpha^2\rho(\bar{a}_0b_0) =
  C_\alpha^2\rho\iota(b_0\bar{a}_0) \\
  \rho\iota\sigma_{2\ast}(\bar{c}_0\bar{a}_0) & =
  \rho\iota(\bar{d}_0\bar{b}_0) = \rho(\bar{a}_0\bar{d}_0\bar{b}_0a_0)
  = \alpha^{-1} \\
  & = \alpha^3 = C_\alpha^2\rho(\bar{a}_0\bar{c}_0) =
  C_\alpha^2\rho\iota(\bar{c}_0\bar{a}_0) \\
  \rho\iota\sigma_{2\ast}(\bar{d}_0\bar{a}_0) & =
  \rho\iota(\bar{c}_0\bar{b}_0) = \rho(\bar{a}_0\bar{c}_0\bar{b}_0a_0)
  = \alpha \\
  & = C_\alpha^2\rho(\bar{a}_0\bar{d}_0) =
  C_\alpha^2\rho\iota(\bar{d}_0\bar{a}_0) \\
  \rho\iota\sigma_{2\ast}(a_0a_1\bar{b}_1\bar{a}_0) & =
  \rho\iota(b_0b_1\bar{a}_1\bar{b}_0) =
  \rho(\bar{a}_0b_0b_1\bar{a}_1\bar{b}_0a_0) =
  \alpha^2\beta^{-2}\alpha^{-2} \\
  & = \alpha^2\beta^2\alpha^{-2} = C_\alpha^2\rho(a_1\bar{b}_1) =
  C_\alpha^2\rho\iota(a_0a_1\bar{b}_1\bar{a}_0) \\
  \rho\iota\sigma_{2\ast}(a_0a_1c_1\bar{a}_0) & =
  \rho\iota(b_0b_1d_1\bar{b}_0) = \rho(\bar{a}_0b_0b_1d_1\bar{b}_0a_0)
  = \alpha^2\beta^{-1}\alpha^{-2} \\
  & = \alpha^2\beta^3\alpha^{-2} = C_\alpha^2\rho(a_1c_1) =
  C_\alpha^2\rho\iota(a_0a_1c_1\bar{a}_0) \\
  \rho\iota\sigma_{2\ast}(a_0a_1d_1\bar{a}_0) & =
  \rho\iota(b_0b_1c_1\bar{b}_0) = \rho(\bar{a}_0b_0b_1c_1\bar{b}_0a_0)
  = \alpha^2\beta\alpha^{-2} \\
  & = C_\alpha^2\rho(a_1d_1) = C_\alpha^2\rho\iota(a_0a_1d_1\bar{a}_0)
\end{align*}
and so the claim is established.

The symmetry  in equation (\ref{eq:sigma2}) is key to showing that  $\sigma_3$  lifts to an 
automorphism $\eta$ of $Y$.  First, equation (\ref{eq:sigma2})  is used to  show that 
$\sigma_3 \co \Theta^3 - L \to \Theta^3 - L$ lifts to a symmetry of the $125$-fold  
cover,  and then the isometry $\eta$ is obtained by continuous extension to the 
completion $Y$ of this cover. 

Recall from item (iv) of subsection~\ref{subsec:cover} that the $125$-fold cover of 
$\Theta^3 - L$ corresponds to the subgroup 
$$
H \; = \; \tau^{-1}({\rm Stab}_{S_{125}}((1,1,1)))
$$ 
of $\pi_1(\Theta^3 - L, (0,0,0))$. 
By construction,  $H=H_1 \cap
H_2 \cap H_3$, where
$$
H_i=\rho_i^{-1}(\rm{Stab}_{S_5}(1)), \qquad \text{for} \; \;  1 \leq i \leq 3.
$$
The isometry $\sigma_3$ lifts to this cover provided that $\sigma_{3\ast}(H) \subset H$. 
This will follow if $\sigma_{3\ast}(H_i) \subset H_i$ for $1 \leq i \leq 3$. 
If $h \in H_i$, (i.e. $\rho_i (h) \in \rm{Stab}_{S_5}(1)$) then
\begin{align*}
  \rho_i\sigma_{3\ast} (h)(1) &= \rho \iota {\rm
    pr}_{i\ast}\sigma_{3\ast} (h)(1) &
  \tag*{\scriptsize definition of $\rho_i$}\\
  &= \rho \iota \sigma_{2\ast} {\rm{pr}}_{i\ast} (h)(1) &
  \tag*{\scriptsize equation~\eqref{eq:projection}}\\
  &= C_\alpha^2 \rho \iota {\rm{pr}}_{i\ast} (h)(1)
  \tag*{\scriptsize equation~\eqref{eq:sigma2}}\\
  &= C_\alpha^2\rho_i(h)(1) &
  \tag*{\scriptsize definition of $\rho_i$}\\
  &= \alpha^2\rho_i(h)\alpha^{-2}(1) &
  \tag*{\scriptsize definition of $C_\alpha$} \\
  &= 1 & \tag*{\scriptsize as $\rho_i(h),\alpha \in
    \rm{Stab}_{S_5}(1)$}
\end{align*}
We have shown $\rho_i(\sigma_{3\ast}(h)) \in {\rm
  Stab}_{S_5}(1)$. Thus $\sigma_{3\ast} (H_i) \subset H_i$, and so
$\sigma_3$ lifts to the $125$-fold cover of $\Theta^3 - L$.

Since $\sigma_3$ fixes the vertex $(0,0,0) \in \Theta^3 - L$ and acts
freely on the link of $(0,0,0)$, we can choose a lift
$\widehat{\sigma}_3$ to the $125$-fold cover which fixes a vertex $v$
in the fiber over $(0,0,0)$, and acts freely on the link of this
vertex.  Thus the continuous extension $\eta\co Y \to Y$ of
$\widehat{\sigma}_3$ is an isometry which fixes a vertex of
$b^{-1}((0,0,0)) \subset Y$, and which acts freely on the link of this
vertex.
 
The square $\widehat{\sigma}_3^2$ also fixes $v$ and, covers the
identity map $\sigma_3^2 \co \Theta^3 - L \to \Theta^3 - L$. Thus
$\widehat{\sigma}_3^2$ is a deck transformation of the $125$-fold
cover of $\Theta^3 - L$ which fixes a point, and so is the identity
map. Its continuous extension $\eta^2$ is the identity map on $Y$, and
so $\eta$ is an order two isometry of $Y$.

Let $K$ be the kernel of the map $\pi_1(Y) \to \pi_1(S^1)$.  By
Remark~\ref{rk:property} $\pi_1(Y)$ is hyperbolic, and $K$ is finitely
presented but not of type F$_3$.  Since $\eta$ has finite order, the
subgroup $K \rtimes \la \eta \ra \subset \pi_1(Y) \rtimes \la \eta
\ra$ is a finitely presented subgroup of the hyperbolic group
$\pi_1(Y) \rtimes \la \eta \ra$ which has infinitely many conjugacy
classes of finite-order elements. The conjugacy class count follows
either from Proposition~\ref{prop:conj} above, or from Proposition~1.1
of \cite{BrCD}.

%%%%%%%%%%%%%%%%%%%%%%%%%%%%%%%%%%%%%%%%%%%%%%%%%%%%%%%%%%%%%%%%%%%%%%%%%%%%

\bibliography{bibliography}

\def\cprime{$'$}
\begin{thebibliography}{1}

\bibitem{BeBr}
{\sc M.~Bestvina and N.~Brady}, {\em Morse theory and finiteness properties of
  groups}, Invent. Math., 129 (1997), pp.~445--470.

\bibitem{Br}
{\sc N.~Brady}, {\em Branched coverings of cubical complexes and subgroups of
  hyperbolic groups}, J. London Math. Soc. (2), 60 (1999), pp.~461--480.

\bibitem{BrCD}
{\sc N.~Brady, M.~Clay, and P.~Dani}, {\em Morse theory and conjugacy classes
  of finite subgroups}, Geom. Dedicata, 135 (2008), pp.~15--22.

\bibitem{BrMi}
{\sc N.~Brady and A.~Miller}, {\em {$\rm CAT(-1)$} structures for free-by-free
  groups}, Geom. Dedicata, 90 (2002), pp.~77--98.

\bibitem{Bridson00}
{\sc M.~R. Bridson}, {\em Finiteness properties for subgroups of {${\rm
  GL}(n,\mathbf Z)$}}, Math. Ann., 317 (2000), pp.~629--633.

\bibitem{brown}
{\sc K.~S. Brown}, {\em Finiteness properties of groups}, in Proceedings of the
  {N}orthwestern conference on cohomology of groups ({E}vanston, {I}ll., 1985),
  vol.~44, 1987, pp.~45--75.

\bibitem{feighn-mess}
{\sc M.~Feighn and G.~Mess}, {\em Conjugacy classes of finite subgroups of
  {K}leinian groups}, Amer. J. Math., 113 (1991), pp.~179--188.

\bibitem{LN}
{\sc I.~J. Leary and B.~E.~A. Nucinkis}, {\em Some groups of type {$VF$}},
  Invent. Math., 151 (2003), pp.~135--165.

\bibitem{Wise}
{\sc D.~T. Wise}, {\em Incoherent negatively curved groups}, Proc. Amer. Math.
  Soc., 126 (1998), pp.~957--964.

\end{thebibliography}
\bibliographystyle{siam}

\end{document}